\documentclass{amsart}
\usepackage[utf8]{inputenc}
\usepackage{amssymb,bm,amsfonts, stmaryrd, amsmath, amsthm, enumerate, mathtools,comment}
\usepackage[all]{xy}
\makeatletter
\@namedef{subjclassname@2020}{\textup{2020} Mathematics Subject Classification}
\makeatother
\usepackage[colorlinks]{hyperref}
\hypersetup{
 colorlinks=true,
 linkcolor=blue,
 filecolor=magenta, 
 urlcolor=cyan,
}
\usepackage{wrapfig,tikz, tikz-cd}
\usetikzlibrary{arrows}
\usepackage[capitalise]{cleveref}
\crefformat{equation}{(#2#1#3)}
\crefrangeformat{equation}{(#3#1#4--#5#2#6)}
\crefmultiformat{equation}{(#2#1#3}{ \& #2#1#3)}{, #2#1#3}{ \& #2#1#3)}
\crefformat{enumi}{(#2#1#3)}
\crefrangeformat{enumi}{(#3#1#4--#5#2#6)}
\newtheorem{theorem}{Theorem}[section]
\newtheorem{lemma}[theorem]{Lemma}
\newtheorem{proposition}[theorem]{Proposition}
\newtheorem{corollary}[theorem]{Corollary}
\newcounter{intro}
\newtheorem{introconjecture}[intro]{Conjecture}
\newtheorem{introthm}[intro]{Theorem}
\newtheorem{introcor}[intro]{Corollary}

\theoremstyle{definition}

\newtheorem{example}[theorem]{Example}
\newtheorem{remark}[theorem]{Remark}

\newtheorem{conjecture}[theorem]{Conjecture}
\newtheorem{notation}[theorem]{Notation}

\newtheorem{chunk}[theorem]{}

\newtheorem*{ack}{Acknowledgements}

\newtheorem{question}[theorem]{Question}

\newcommand{\ges}{\geqslant}
\newcommand{\les}{\leqslant}

\newcommand{\Hom}{{\operatorname{Hom}}}

\newcommand{\Tor}{{\operatorname{Tor}}}

\DeclareMathOperator{\cone}{cone}
\newcommand{\del}{\partial}
\renewcommand{\H}{\operatorname{H}}

\DeclareMathOperator{\pdim}{projdim}
\DeclareMathOperator{\fdim}{flatdim}

\DeclareMathOperator{\depth}{depth}

\DeclareMathOperator{\reg}{reg}

\DeclareMathOperator{\ord}{ord}
\DeclareMathOperator{\cx}{cx}

\newcommand{\xra}{\xrightarrow}

\newcommand{\vp}{\varphi}

\newcommand{\x}{\bm{x}}

\newcommand{\g}{\mathsf{g}}

\renewcommand{\ll}{\ell\ell}
\newcommand{\gll}{g\ell\ell}
\DeclareMathOperator{\gr}{gr}

\newcommand{\ann}{{\operatorname{ann}}}

\newcommand{\shift}{{\scriptstyle\mathsf{\Sigma}}}
\newcommand{\lotimes}{\otimes^{\sf L}}

\DeclareMathOperator{\rank}{rank}

\newcommand{\m}{\mathfrak{m}}
\newcommand{\n}{\mathfrak{n}}

\title[Lower bounds on Loewy length]{Lower bounds on Loewy lengths of modules of finite projective dimension}

\author[N.~KC]{Nawaj KC}
\address{Department of Mathematics,
University of Nebraska, Lincoln, NE 68588, U.S.A.}
\email{nkc3@huskers.unl.edu}
\author[J.~Pollitz]{Josh Pollitz}
\address{Mathematics Department, 
Syracuse University, 
Syracuse, NY 13244 U.S.A.}
\email{jhpollit@syr.edu}


\begin{document}

\begin{abstract}
This article is concerned with nonzero modules of finite length and finite projective dimension over a local ring. We show the Loewy length of such a module is larger than the regularity of the ring whenever the ring is strict Cohen-Macaulay, establishing a conjecture of Corso--Huneke--Polini--Ulrich for such rings. In fact, we show the stronger result that the Loewy length of a nonzero module of finite flat dimension is at least the regularity for strict Cohen-Macaulay rings, which significantly strengthens a theorem of Avramov--Buchweitz--Iyengar--Miller.
As an application we simultaneously verify a Lech-like conjecture, comparing generalized Loewy length along flat local extensions, and a conjecture of Hanes for strict Cohen-Macaulay rings. Finally, we also give notable improvements to known lower bounds for Loewy lengths without the strict Cohen-Macaulay assumption. The strongest general bounds we achieve are over complete intersection rings. 
\end{abstract}

\maketitle

\section*{Introduction}\label{s_intro}

This work is concerned with modules of finite length and finite projective dimension over a local ring. Such modules have received a great deal of attention and encode the singularity of the ring; for instance, a consequence of Roberts' New Intersection Theorem~\cite{Roberts:1987} is that if a  ring admits a nonzero module of finite length and finite projective dimension, then the ring must be Cohen-Macaulay. There are a number of interesting questions that remain open regarding these modules. In this article we revisit one on a uniform lower bound for their Loewy lengths. 

Fix a local ring $R$ with maximal ideal $\m$.  Recall that the Loewy length of an $R$-module $M$ is 
$\ll_R(M)=\inf\{i\ges 0\mid \m^i M=0\};$ this is also the minimal length filtration of $M$ with semisimple subquotients. This article focuses on understanding
\[\inf\{\ll_R(M)\mid M\neq 0\text{ finitely generated and }\pdim_R(M)<\infty\}\,,\]
which (by the New Intersection Theorem) is finite if and only if $R$ is Cohen-Macaulay; furthermore, the Auslander--Buchsbaum--Serre theorem asserts this infimum is one when $R$ is regular and at least two when $R$ is singular. Hence, this value can be interpreted as a numerical measure of the singularity of the ring. 

For a Cohen-Macaulay ring, the simplest modules to produce that have both finite length and projective dimension are quotients by systems of parameters. The following conjecture of Corso, Huneke, Polini and Ulrich~\cite{Corso/Huenke/Polini/Ulrich}, known as the \emph{Loewy Length Conjecture}, predicts the infimum above is achieved by one of these modules. 

\begin{introconjecture}\label{introconj}
    For a local ring $R$ and a nonzero $R$-module $M$ with finite projective dimension, the following inequality holds: \[\ll_R(M) \ges  \min\{\ll_R(R/(\x))\mid \x\text{ is a system of parameters on }R\}\,.\]
\end{introconjecture}

The proposed uniform lower bound in the conjecture is called the generalized Loewy length of $R$, denoted $\gll_R(R)$; this invariant is a Loewy length analog to the Hilbert-Samuel multiplicity of $R$, and it has been studied in a number of works; see, for example, \cite{Avramov/Buchweitz/Iyengar/Miller:2010,Briggs/Grifo/Pollitz:2024,Corso/Huenke/Polini/Ulrich,Alessandro:2016,Ding:1994,Hashimoto/Shida:1997,Koh/Lee:1998,Puthenpurakal:2017}.
 
 The strongest evidence for the Loewy Length Conjecture was established by  Avramov--Buchweitz--Iyengar--Miller~\cite[Theorem~1.2]{Avramov/Buchweitz/Iyengar/Miller:2010}: \emph{If $R$ is a Gorenstein ring with Cohen-Macaulay associated graded ring $R^\g$ and infinite residue field, then \cref{introconj} holds}. Without any assumption on the residue field, they also provide a bound in terms of the (Castelnouvo-Mumford) regularity of $R^\g$, denoted $\reg(R^\g)$. The main result of this article is the following significant strengthening.

\begin{introthm}
\label{t_main_intro}
    If $R$ is local with Cohen-Macaulay associated graded ring, then any nonzero $R$-module $M$ of finite flat dimension satisfies:
    \[\ll_R(M)\geqslant \reg(R^\g)+1\,.\]
    Furthermore, if $R$ has an infinite residue field, then $\ll_R(M)\geqslant \gll_R(R).$
\end{introthm}

It is worth highlighting that the result above applies to modules that are not necessarily finitely generated, and that both the finite generation and Gorenstein properties were fundamental ingredients in the proof of the result of Avramov--Buchweitz--Iyengar--Miller. Furthermore, the bounds in \cref{t_main_intro} were not previously known even for (localizations of) standard graded Cohen-Macaulay algebras. 

The proof of \cref{t_main_intro} can be found at the end of \cref{s_main_theorem}. The two essential ingredients are \cref{newlemma} and \cref{mainConstruction}. The former abstractly identifies certain modules whose existence impose lower bounds on the Loewy lengths of nonzero modules having finite flat dimension; the latter is where the strict Cohen-Macaulay assumption is leveraged to show that these special modules (specified in \cref{newlemma}) exist over such rings and are used to establish the bounds in \cref{new_main_theorem}.

There are several consequences of \cref{t_main_intro} to strict Cohen-Macaulay rings. One is the following, where the generalized Loewy length of a module is defined analogous to how it was defined for the ring; cf.\@ \cref{Loewy_length}. 

\begin{introcor}
     With $R$ as in \cref{t_main_intro}, if $M$ is a nonzero Cohen-Macaulay $R$-module of finite projective dimension, then  
    \[ \gll_R(M) \ges \reg(R^\g)+1\,. \]
    Furthermore, if $R$ has an infinite residue field, then $\gll_R(M)\ges \gll_R(R).$ 
\end{introcor}

Another application is that when $R$ is strict Cohen-Macaulay,
\[\gll_R(R) = \ll_R(R/(\x))\] for all maximal superficial sequences $\x$ in $\m\smallsetminus\m^2$, where $\m$ denotes the maximal ideal of $R$;  cf.\@ \cref{perfectModules}. Furthermore, we introduce a version of the long-standing conjecture of Lech~\cite{Lech:1960}, where Hilbert-Samuel multiplicity is replaced with generalized Loewy length (see \cref{conjecture_loewy_lech}), and we provide the following evidence:

\begin{introcor}
\label{intro_cor_gll}
    Suppose $R\to S$ is a map of Cohen-Macaulay local rings with infinite residue fields. If $\fdim_R(S)<\infty$ and $R$ is strict Cohen-Macaulay, then \[\gll_R(R)\les \gll_S(S)\,.\]
\end{introcor}

 This corollary establishes interesting ring theoretic properties along flat local extensions whose base is strict Cohen-Macaulay; cf.\@ \cref{cor_reduction_regularity_nondec}. In particular, we verify a conjecture of Hanes~\cite[Conjecture~3.1]{Hanes:2001} on reduction numbers along flat local extensions when the base is assumed to be strict Cohen-Macaulay.

In the final part of the paper, \cref{s_ci}, we turn our attention to  lower bounds for the Loewy lengths of nonzero modules of finite flat dimension without making assumptions on the structure of $R^\g$. We give a number of general results that are in terms of the $\m$-adic order, denoted ord, of generators for the defining relations of $R$. The main result in this section is the following.

 \begin{introthm}
 \label{t_intro_ci}
  \label{intro_ci}
     If $R$  is complete intersection of codimension $c$, and $M$ is a nonzero $R$-module of finite flat dimension, then 
          \[\ll_R(M)\ges \sum_{i=1}^c\ord(f_i)-c+1\,,\]
          where $\widehat{R}\cong Q/(f_1,\ldots,f_c)$ is any Cohen presentation of $R$. 
 \end{introthm}

We present two proofs of the theorem above. The first is similar to the proof of \cref{t_main_intro} as it constructs certain artinian quotients so that \cref{newlemma} can be applied, while the second was sketched to the authors by Walker (at least when $M$ has finite length) and makes use of a construction from~\cite{Herzog/Ulrich/Backelin:1991}. In many regards, the lower bound in the previous theorem could be regarded as the ``expected" generalized Loewy length of a complete intersection ring. However, without assuming $R^\g$ is complete intersection itself, this expected generalized Loewy length can be arbitrarily smaller than the actual generalized Loewy length of $R$; see \cref{e_nonstrictci} for strict Cohen-Macaulay complete intersection rings witnessing this behavior.  

Finally, the central problem in this article (\cref{introconj}) complements the \emph{Length Conjecture} studied in \cite{Iyengar/Ma/Walker:2022}; cf.\@ \cref{c_IMW}. As Loewy length and length are subtle invariants in their own way, it is perhaps unsurprising we employ different techniques in the present article to establish our uniform lower bounds for Loewy lengths compared to the methods used by Iyengar--Ma--Walker for length. It is worth remarking that the existence of Ulrich modules, which played a pivotal role in \cite{Iyengar/Ma/Walker:2022}, are now known to not exist over rings where the Loewy Length Conjecture was settled in the positive here. Namely, \cite{Iyengar/Ma/Walker/Zhuang:2024} provides examples of complete intersection rings that are strict Cohen-Macaulay where the method from \cite{Iyengar/Ma/Walker:2022} cannot be applied; albeit, the Length Conjecture remains unsettled for these examples.

\begin{ack}
We  thank Srikanth Iyengar for several helpful discussions regarding this work and comments on previous versions of the manuscript. 
We also thank Mark Walker for suggesting a proof of \cref{t_mark} with us, as well as Linquan Ma for making us aware of the conjecture in \cite{Hanes:2001}. We are grateful to Alberto Corso,  Craig Huneke, Claudia Polini and Bernd Ulrich for discussing their conjectures, and their upcoming work in \cite{Corso/Huenke/Polini/Ulrich}, with us. Finally, we are very thankful to the referee, their many comments led to significant improvements of the article.

The first author was supported by the National Science Foundation grant DMS-1928930 and by the Alfred P.\@ Sloan Foundation under grant G-2021-16778, while he was in residence at the Simons Laufer Mathematical Sciences Institute (formerly MSRI) in Berkeley, California, during the Spring 2024 semester. The first author was also partly supported by NSF grants DMS-2200732 and DMS-2044833. The second author was supported by the NSF grant  DMS-2302567.
\end{ack}

\section{Loewy length and the associated graded ring}\label{s_gll_sr}

This section recounts the necessary background on (generalized) Loewy length. Throughout, $R$ is a commutative noetherian local ring with maximal ideal $\m$. 

\begin{chunk}
\label{Loewy_length}
The \emph{Loewy length} of an $R$-module $M$, denoted $\ll_R(M)$, is the least nonnegative integer $i$ such that  $\m^i M = 0$; if no such integer exists set $\ll_R(M)=\infty$. That is to say, $\ll_R(M)$ is the infimum over all nonnegative integers $i$ such that $M$ admits a filtration of length $i$ with subquotients semisimple $R$-modules (i.e., each subquotient is a direct sum of copies of $k$). The Loewy length of a module is always bounded above by its length, but there are plenty of infinitely generated modules of finite Loewy length. For a finitely generated module, its length is finite if and only if its Loewy length is finite. 

When $M$ is a finitely generated $R$-module, its \emph{generalized Loewy length} is
   \[ \gll_R(M) = \inf\{ \ll_R(M/\x M) \mid \x\text{ is a system of parameters on } M \}.\]
   The invariant $\gll_R(R)$ can be regarded as a measure of the singularity of $R$, as it equals one if and only if $R$ is regular. 
\end{chunk}

\begin{remark}\label{r_loewy}
    Generalized Loewy length is to Loewy length as Hilbert-Samuel multiplicity is to length. Whereas there is much known on the Hilbert-Samuel multiplicity of modules over a local ring, generalized Loewy length can be a more subtle invariant; for instance, Loewy length can fail to be additive even along split exact sequences. Another example, relevant to this article, is that the Hilbert-Samuel multiplicity of a Cohen-Macaulay local ring $R$ is the length of $R/(\x)$ where $\x$ is a sufficiently general system of parameters; whether generalized Loewy length is the Loewy length of $R/(\x)$ for such a $\x$ remains an unanswered question of DeStefani over Gorenstein rings having infinite residue fields \cite[Question~4.5]{Alessandro:2016}. 
\end{remark}

\begin{notation}
    The associated graded ring of $R$ is
\[R^\g \coloneqq\gr_{\m}(R)=\bigoplus_{j=0}^\infty \frac{\m^j}{\m^{j+1}}\,.\] For a sequence of elements $\x$ in $R$, write $\x^*$ for its corresponding sequence of initial forms in $R^\g.$
 Also, if $C$ is a minimal $R$-complex, then it admits an exhaustive filtration by subcomplexes $\{\mathcal{F}^iC\}_{i\geqslant 0}$ where 
 \[
(\mathcal{F}^iC)_j=\m^{j-i} C \quad \text{ for each }j\in \mathbb{Z}\,.
\]
The \emph{linear part} of $C$ is the associated graded complex of $C$ with respect to this filtration, denoted $C^\g$; cf.\@ \cite[1.2]{Herzog/Iyengar:2005}. 
In particular $C^\g$ is bigraded with 
\[
C^\g_{i,j}= \frac{\m^{j-i}C_{i}}{\m^{j-i+1}C_{i}}\,,
\]
and the entries in each differential are linear forms of $R^\g$. 
\end{notation}

\begin{chunk}\label{c_ll_reduction}
     A sequence $\x$ in $R$ is \emph{superficial} if its sequence of initial forms in $R^\g$ is part of a system of parameters for $R^\g$; such a sequence is part of a system of parameters for $R$. Moreover, if $R$ has an infinite residue field, then there exists a superficial system of parameters in $\m\smallsetminus \m^2$ \cite[Theorem 8.6.6]{Huneke/Swanson:2006}. 
\end{chunk}

\begin{chunk}
\label{c_strict_ci_sr}
A local ring $R$ is \textit{strict Cohen-Macaulay} if its associated graded ring $R^\g$ is Cohen-Macaulay. Strict Cohen-Macaulay rings are, in particular, Cohen-Macaulay; see, for example, \cite{Achilles/Avramov:1982} (or \cite{Froberg:1987}). Hypersurface rings are strict Cohen-Macaulay, however  complete intersection rings of higher codimension need not be strict Cohen-Macaulay (see, for example, \cref{r_alessandro}).
\end{chunk}

\begin{chunk} \label{c_regularity}
For a standard graded algebra $S$ over a field $S_0$, write $S_+$ for its homogeneous maximal ideal. For a homogeneous ideal $I$ of $S$, let $\H_I(-)$ denote the local cohomology functor with respect to $I$. The \emph{regularity of} $S$ is 
\[
\reg(S)\coloneqq \max\{i+j\mid \H_{S_+}^i(S)_j\neq 0\text{ for some }i\}\,.
\]
\end{chunk}

\begin{lemma}
\label{l_bound}
    If $R$ is a local ring with a superficial system of parameters $\x$ in $\m\smallsetminus\m^2$, and set $A=\mathrm{Kos}^R(\x)$,  then 
    \[\ll_R(R/(\x))-1\les \reg(R^\g)=\max\{j\ges 0\mid \H_i(A^\g)_{i+j}\neq 0 \text{ for some }i\}\,.\]
\end{lemma}
\begin{proof}
    Set $d=\dim(R)$, and fix a Noether normalization $S=k[t_1,\ldots, t_d]\hookrightarrow R^\g$ 
with each $t_i$ mapping to $x_i^*$, the initial form of $x_i$ in $R^\g$. By the assumptions on $\x$, we have
   \[
    \H_{R^\g_+}^i(R^\g)\cong \H_{(\bm{t})}^i(R^\g)
   \]
   and  hence it follows from a theorem of Eisenbud--Goto (see, for example, \cite[Theorem~4.3.1]{Bruns/Herzog:1998}) that
   \[
  \reg(R^\g)=\max\{j\ges 0\mid \H_i(\mathrm{Kos}^{S}(\bm{t}; R^\g))_{i+j}\neq 0\text{ for some }i\}\,.
   \]
   Finally, it remains to observe that  $\mathrm{Kos}^{S}(\bm{t}; R^\g)\cong A^\g$ to justify the desired equality. 

   For the inequality, it follows from  \cite[Proposition~8.2.4]{Huneke/Swanson:2006} that
\[\ll_R(R/(\x))\les \min\{j\mid (R^\g/(\x^*))_{\ges j}= 0\}\,.\]
Now from the already established equality, using that $\H_0(A^\g)=R^\g/(\x^*),$ we obtain the desired inequality. 
\end{proof}

\begin{lemma}
\label{example_strict_cm}
    If $R$ is strict Cohen-Macaulay, then $\reg(R^\g)=\ll_R(R/(\x))-1$ for any superficial system of parameters $\x$ in $\m\smallsetminus\m^2$.  
\end{lemma}
\begin{proof}
By \cite[Lemma~0.1]{Sally:1979}, since $R^\g$ is Cohen-Macaulay, for any superficial sequence $\x$ we have
\[
 (R/(\x))^\g\cong R^\g/(\x^*)\,.
\]
Hence it follows from \cite[Lemma~6.2]{Hashimoto/Shida:1997} and \cref{l_bound} that 
\[\ll_R(R/(\x))= \min\{j\mid (R^\g/(\x^*))_{\ges j}= 0\} = \reg(R^\g)+1\,.\qedhere\]
 \end{proof}

\section{Main result}\label{s_main_theorem}

This section is devoted to proving \cref{t_main_intro} from the introduction. In fact, we show the following stronger result. 

\begin{theorem}\label{new_main_theorem}
        If $R$ is strict Cohen-Macaulay with residue field $k$, then any nonzero $R$-module $M$ with $\Tor^R_{i}(M,k)=0$ for some $i\geqslant 0$ satisfies:
    \[\ll_R(M)\geqslant \reg(R^\g)+1\,.\]
\end{theorem}

A consequence of the theorem affirms two conjectures of Corso--Huneke--Polini--Ulrich, shared with us in private correspondences, for strict Cohen-Macaulay rings. 

\begin{corollary} \label{mainTheorem}
 Suppose a local ring $R$ is strict Cohen-Macaulay. If $M$ is a nonzero Cohen-Macaulay $R$-module of finite projective dimension, then  
    \[ \gll_R(M) \ges \reg(R^\g)+1\,. \]
\end{corollary}

\begin{corollary}\label{perfectModules}
   For a strict Cohen-Macaulay ring $(R,\m)$ with infinite residue field, 
    \[
    \gll_R(R)=\ll_R(R/(\x))=\reg(R^\g)+1
    \]
    with $\x$ a sufficiently general system of parameters for $R$ in $\m\smallsetminus\m^2$. Consequently, if $M$ is a nonzero Cohen-Macaulay $R$-module of finite projective dimension, then \looseness -1
    \[
    \gll_R(M)\geqslant \gll(R)\,.
    \]
\end{corollary}

The proofs of these results are given at the end of the section.  For the remainder of the section we assume $R$ is a commutative noetherian local ring with unique maximal ideal $\m$ and residue field $k.$ The following lemma is one of the major inputs in what follows.

\begin{lemma}
    \label{newlemma}
    Assume $N$ is an $R$-module with finite Loewy length. Fix $F\xra{\simeq} k$ and $G\xra{\simeq} N$, free resolutions over $R$. If there exist nonnegative integers $n,d$ and a lift $\sigma\colon F\to G$ of a nonzero map $k\to N$ with $\sigma(F_i)\subseteq \m^n G_i$ for all $i\leqslant d$, then for any nonzero $R$-module $M$ and $\Tor^R_{d+1}(M,k)=0$  we have $\ll_R(M)\ges n+1.$
\end{lemma}
\begin{proof}
    Let $N'$ denote the cokernel of the nonzero map $k\to N$, and set $C=\cone(\sigma).$ First, we claim that $C$ is a free resolution of $N'$. Indeed, the cone exact sequence 
\[
0\to G\to C\to \shift F\to 0
\]
induces the exact sequence 
\[
 \cdots \to 0 \to 0\to\H_2(C)\to 0\to 0\to \H_1(C)\to \H_0(F)\xra{\H_0(\sigma)}\H_0(G)\to \H_0(C)\to 0\,.
\]
In particular, $\H_i(C)=0$ for $i>1$.
Furthermore, note that $\H_0(\sigma)$ is exactly the given map $k\to N$, and so 
\[
\H_i(C)=\begin{cases} N'& i=0\\ 0 & \text{otherwise}\end{cases}\,.
\]
Finally, as $F$ and $G$ are nonnegatively free graded $R$-modules, we conclude that $C$ is a free resolution of $N'.$

Now assume $M$ is a nonzero $R$-module with $\ll_R(M)\leqslant n$ and $\Tor^R_{d+1}(M,k)=0.$ We can further assume $\Tor^R_d(M,k)\neq 0$. Also, by inducting on Loewy length it follows that  $\Tor^R_{d+1}(M,-)$ vanishes on the category of $R$-modules having finite Loewy length.

As $\ll_R(M)\leqslant n$  and $\sigma(F_i)\subseteq \m^n G_i$ for all $i\leqslant d$, we have the third isomorphism below for $j\leqslant d+1$, while the others are obvious: 
\begin{align*}
    \Tor^R_j(M,N')&\cong\H_j(M\otimes_R C)\\
    &\cong \H_j(\cone(M\otimes_R \sigma)) \\
    &\cong \H_j(\cone(M\otimes_R F\xra{0}M\otimes_R G))\\
    &=\H_j( M\otimes_RG\oplus \shift(M\otimes_R F))\\
    &\cong \H_j(M\otimes_RG)\oplus \shift \H_j(M\otimes_R F)\\
    &\cong\Tor^R_j(M,N)\oplus  \Tor^R_{j-1}(M,k)\,.
\end{align*}
 However, observe that for $j=d+1$ the isomorphisms above yield
\[
\Tor_{d+1}^R(M,N')\cong \Tor^R_{d+1}(M, N)\oplus \Tor_d^R(M,k)\,,
\]
 giving us a contradiction, since $\Tor^R_{d+1}(M,N')=0$ as $N'$ has finite Loewy length. Therefore, $\ll_R(M)\ges n+1$, as desired. 
\end{proof}

We remark that when  getting a lower bound on the Loewy length of a nonzero module of \emph{finite length} and finite projective dimension, one can modify the end of the previous proof to relax the hypothesis on $N$ in \cref{newlemma}; instead of inducting on Loewy length, the Auslander--Buchsbaum formula can be applied. Namely, we  also have the following. 

\begin{lemma}
    \label{newlemma'}
    Fix $F\xra{\simeq} k$ and $G\xra{\simeq} N$, free resolutions over $R$. 
    If there exist nonnegative integers $n,d$ and a lift $\sigma\colon F\to G$ of a nonzero map $k\to N$ satisfying that $\sigma(F_i)\subseteq \m^n G_i$ for all $i\leqslant d$, then for any nonzero finitely generated $R$-module $M$ with $\pdim_R(M)\les d$ we have $\ll_R(M)\ges n+1.$ \qed
\end{lemma}

We note that the assumption in \cref{newlemma} on $M$, when it has finite Loewy length, implies that $M$ has flat dimension at most $d$. Namely, we have the following proposition. Stating \cref{newlemma} in its current form, makes for finitely many things to check when attempting to apply the lemma. 

\begin{proposition}
    If $M$ has finite Loewy length and $\Tor^R_{d+1}(M,k)=0$, then $\mathrm{flatdim}_R(M)\les d.$
\end{proposition}

\begin{proof}
 It suffices to show $\Tor^R_{d+1}(M,-)$ vanishes on the category of finitely generated $R$-modules. We verify this by induction on Krull dimension.  
 
 For the base case, a straightforward induction argument on length yields that $\Tor^R_{d+1}(M,-)$ vanishes on the category of finite length $R$-modules. 
 
 Now if $L$ is a finitely generated $R$-module of positive Krull dimension, the exact sequence
    \[
    0\to \Gamma_\m(L)\to L\to L'\to 0
    \]
    induces an inclusion $\Tor_{d+1}^R(M,L)\hookrightarrow \Tor^R_{d+1}(M,L')$; here we have used that $\Tor^R_{d+1}(M,\Gamma_\m(L))=0$. Therefore, we can assume $L$ has positive depth, and so there exists an $L$-regular element $x$.  From this, and the inductive hypothesis applied to $L/xL$, we obtain the surjective map 
    \[
    \Tor^R_{d+1}(M,L)\xra{x\cdot}\Tor^R_{d+1}(M,L)\,,
    \]
    and so $\Tor^R_{d+1}(M,L)\subseteq \m \Tor^R_{d+1}(M,L)$. However, since $\Tor^R_{d+1}(M,L)$ has finite Loewy length it follows that $\Tor^R_{d+1}(M,L)=0.$
\end{proof}

The next lemma abstracts a discussion in \cite[Section~2]{Sega:2013}, building on work in \cite{Herzog/Iyengar:2005}. It is also related to the main theorem of \cite{Trung:1998}, as well as \cite[Section~3]{Avramov/Iyengar/Miller:2006}.

\begin{lemma}
     \label{p_linearitydefect}
    Suppose $(R,\m)$ is a local ring, $\x$ is a superficial system of parameters in $\m\smallsetminus \m^2$, and set $A=\mathrm{Kos}^R(\x)$. Then for $i\ges 0$ and $j \ges \reg(R^\g)$ we have 
    \[ \partial(A_{i+1}) \cap \m^{j+1}A_{i} = \partial(\m^{j} A_{i+1})\,. \] 
\end{lemma}

\begin{proof}
There is no harm in completing $R$ to show the desired equality. 
As $A$ is a minimal complex, it is clear that $\partial(\m^{j} A_{i+1})\subseteq \partial(A_{i+1}) \cap \m^{j+1}A_{i}. $ 

For the reverse containment, consider $\del x\in \partial(A_{i+1}) \cap \m^{j+1}A_{i}.$ Its image in $A^\g$ is a cycle in $A^\g_{i,i+j+1}$, and so it defines the following homology class
\[[\del x]\in \H_i(A^\g)_{i+j+1}\,.\]
Since $j\ges \reg(R^\g)$,  by \cref{l_bound}, the homology class $[\del x]$ is zero. That is to say, there exists $u_0\in \m^{j}A_{i+1}$ such that 
\[
\del^A(x)-\del^A(u_0)=\del^A(x-u_0)\in \m^{j+2}A_i\,.
\]
Repeating the argument above produces $u_n\in \m^{j+n}A_{i+1}$ with 
\[
\del^A(x)-\del^A\left(\sum_{\ell=0}^nu_\ell\right)=\del^A\left(x-\sum_{\ell=0}^nu_\ell\right)\in \m^{j+n+2}A_i\,.
\]
Since $R$ is complete, we can let $u=\sum_{\ell=0}^\infty u_\ell$, and by construction $\del^A(x)=\del^A(u)$ with $u\in \m^{j}A_{i+1}$.
\end{proof}

    For the remainder of the section we can assume the residue field $k$ is infinite; the fact that Loewy length decreases and regularity is unchanged among passage to an infinite residue field is due to \cite[Proposition~6.3]{Hashimoto/Shida:1997} and that local cohomology is invariant under flat base change, respectively. We further assume that $R$ is Cohen-Macaulay, as this is the setting of \cref{new_main_theorem}.

 \begin{chunk} 
    Assuming the setup above, there exists a superficial system of parameters $\x$ in $\m\smallsetminus\m^2$; moreover, as $R$ is Cohen-Macaulay, $\x$ is a maximal regular sequence of $R$.  Set $\overline{R}=R/(\x)$ and let $A$ be the Koszul complex on $\x$ over $R$. Note that $A$ is a minimal free resolution of $\overline{R}$ over $R$. Also, fix a minimal free resolution $F\xra{\simeq} k$ over $R$.
    
    Let $\bar s$ denote a nonzero element of highest $\m \overline{R}$-adic order in the socle of $\overline{R}$; cf. \cref{c_order}. That is to say, $\bar s$ is a nonzero element in  $\m^n\overline{R}$ where the Loewy length of $\overline{R}$ is $n+1$.  Moreover, a standard lifting property guarantees the existence of a lift between the complexes of $R$-modules below: 
    \begin{center}
        \begin{tikzcd}
	F & A \\
	k & {\overline R}
	\arrow["\simeq"', from=1-1, to=2-1]
	\arrow["1\mapsto {\bar s}", from=2-1, to=2-2]
	\arrow["\exists \ \sigma", dashed, from=1-1, to=1-2]
	\arrow["\simeq", from=1-2, to=2-2]
\end{tikzcd}\,.
    \end{center}
In what follows we construct a particular lift $\sigma \colon F\to A$ with $\sigma(F)\subseteq \m^n A$ provided that $\reg(R^\g)=\ll_R(\overline{R})-1$. 
\end{chunk}

\begin{lemma}
\label{p_construction}
\label{mainConstruction}
Assume $\reg(R^\g)=\ll_R(R/(\x))-1$ for some superficial system of parameters $\x$ in $\m\smallsetminus\m^2$. Set  $n=\reg(R^\g)$ and let $F$ denote a minimal $R$-free resolution of $k$. Also, let $\overline{R}=R/(\x)$ whose minimal free resolution over $R$ is  $A=\mathrm{Kos}^R(\x)$.

For any nonzero $\overline{s}\in \m^n\overline{R}$, the  $R$-module map $k\mapsto \overline{R}$ given by $1\mapsto \overline{s}$ admits a lift $\sigma\colon F\to A$  with $\sigma(F)\subseteq \m^n A$.
\end{lemma}
\begin{proof}
First, we inductively construct $\sigma$. In degree zero, let $\sigma_0\colon F_0=R\to R=A_0$ be given by multiplication by $s$ where the image of $s$ in $\overline{R}$ is $\bar s$ and $s\in \m^n.$ Now assume we have constructed $\sigma_0,\ldots,\sigma_{i}$ with 
 \[
 \sigma_j(F_j)\subseteq \m^n A_j\quad \text{and}\quad \sigma_{j-1}\del^F_j=\del^A_{j-1} \sigma_j
 \]
 for each $j=0,\ldots,i.$ Observe that 
 \[\sigma_{i}\del_{i+1}^F(F_{i+1})\subseteq \m^{n+1}A_{i}\cap \ker\del^A_{i}=\m^{n+1}A_{i}\cap \del^A(A_{i+1})
 \]
 where the first containment holds by the inductive hypothesis and the equality holds since $A$ is exact. As a consequence, combined with \cref{p_linearitydefect}, we obtain 
 \[
 \sigma_{i}\del_{i+1}^F(F_{i+1})\subseteq \del^A(\m^{n}A_{i+1})\,;
 \]
this is where the assumption that $\reg(R^\g)=\ll_R(R/(\x))-1$ is used. 
Hence by the  lifting property of a free $R$-module, there exists an $R$-linear map $\sigma_{i+1}\colon F_{i+1}\to A_{i+1}$ such that 
 \[
 \sigma_{i}\del^F_{i+1}=\del^A_{i+1}\sigma_{i+1}\quad\text{and}\quad \sigma_{i+1}(F_{i+1})\subseteq \m^n A_{i+1}\,.
 \]
 This completes the construction of the lift $\sigma\colon F\to A$ satisfying $\sigma(F)\subseteq \m^nA$.
\end{proof}

We are now ready to prove the main results of the section. 

\begin{proof}[Proof of \cref{new_main_theorem}]
    As mentioned above, we can assume $R$ is strict Cohen-Macaualay with infinite residue field $k$. Hence,  $\reg(R^\g)=\ll_R(R/(\x))-1$ for some (in fact, any) superficial system of parameters $\x$ in $\m\smallsetminus\m^2$; cf.\@ \cref{c_ll_reduction,example_strict_cm}. Set $\overline{R}=R/(\x)$, and let $A=\mathrm{Kos}^R(\x)$, its minimal $R$-free resolution, and let $F$ denote a minimal $R$-free resolution of $k$. By \cref{p_construction}, there exists a chain map $\sigma\colon F\to A$, lifting a nonzero map of $k\to \overline{R}$, with $\sigma(F)\subseteq \m^{\reg(R^\g)} A$.

    Now we are exactly in the context \cref{newlemma}. Hence, for any nonzero  $R$-module $M$ with $\Tor^R_{i}(M,k)=0$ for some $i\ges 0$ (for example, if $M$ has finite flat dimension) we can apply \cref{newlemma} to conclude that 
    \[
    \ll_R(M)\ges\reg(R^\g)+1\,.\qedhere
    \]
\end{proof}

\begin{proof}[Proof of \cref{mainTheorem}]

Assume that $M$ is a nonzero Cohen-Macaulay module having finite projective dimension over the strictly Cohen-Macaulay ring $R$. For any system of parameters $\x$ on $M$, the $R$-module $M/(\x) M$ is a nonzero finite length module of finite projective dimension. Now one need only apply  \cref{new_main_theorem} to $M/(\x) M$ to deduce 
    \[\ll_R(M/(\x)M)\ges \reg(R^\g)+1\,,\]
    and when $R$ has infinite residue field $\ll_R(M/(\x)M)\ges \gll_R(R)$. Since the inequalities above hold for all system of parameters $\x$ on $M$, we obtain the desired inequalities. 
\end{proof}

\begin{proof}[Proof of \cref{perfectModules}]
    As the residue field is infinite, a sufficiently general system of parameters $\x$ in $\m\smallsetminus\m^2$ is, in particular, superficial; see \cite[Theorem 8.6.6]{Huneke/Swanson:2006}. Since $R$ is strict Cohen-Macaulay, from \cref{example_strict_cm} we have $\ll_R(R/(\x)) = \reg(R^\g)+1$ and from the inequality in \cref{mainTheorem} (applied for $M=R$), we obtain \[\gll_R(R) = \ll_R(R/(\x))=\reg(R^\g)+1\,.\qedhere\]
\end{proof}

\begin{remark}
\label{remark_min_reg}
 For a local ring $(R,\m)$, we say $R$ has  \emph{minimal regularity} if there exists a faithfully flat extension $R \to (S,\mathfrak{n})$ with $S/\m S$ a field and \[\reg(S^\g)=\ll_{S}(S/(\x))-1\] for some superficial system of parameters $\x$ in $\mathfrak{n}\smallsetminus \mathfrak{n}^2$.   Tracking through the proof of \cref{new_main_theorem} and noting \cite[Proposition 6.3]{Hashimoto/Shida:1997}, one sees that the conclusion of \cref{new_main_theorem} holds for Cohen-Macaulay rings having minimal regularity; this applies to its corollaries as well. It is easy to provide examples of rings having minimal regularity that are not strict Cohen-Macaulay. However, we do not know examples of \emph{Cohen-Macaulay} local rings having minimal regularity that are not strict Cohen-Macaulay. \looseness -1

 When $R$ is Cohen-Macaulay with an infinite residue field, having minimal regularity is an intrinsic property. That is to say, a local ring $(R,\m)$ with infinite residue field has minimal regularity if and only if $\reg(R^\g)=\ll_R(R/(\x))-1$ for some  superficial system of parameters $\x$ in $\m\smallsetminus\m^2$. Indeed, the backwards direction is clear, while the forward direction is actually a consequence of the fact that the conclusions of \cref{new_main_theorem} hold for Cohen-Macaulay rings of minimal regularity. 
\end{remark}
\section{Generalized Loewy length along flat extensions}\label{s_loewy_conjecture}

 A large amount of research (see, for example, \cite{Huneke/Ma/Quy/Smirnov:2020,Ma:2014, Ma:2017, Ma:2023} and the references therein) in commutative algebra has been motivated by the longstanding conjecture of Lech~\cite{Lech:1960}: For a flat local extension $R\to S$, one has the following inequality on Hilbert-Samuel multiplicities:
 \[e(R)\les e(S)\,.\]
 Ma has established the conjecture for equicharacterisitc rings of dimension at most three \cite{Ma:2017}, and for all standard graded algebras localized at their homogeneous maximal ideal \cite{Ma:2023}. In light of \cref{r_loewy} and Lech's conjecture,  we are led to (perhaps optimistically) conjecture the following.

 \begin{conjecture}\label{conjecture_loewy_lech}
     If $R\to S$ is a flat local extension between Cohen-Macaulay rings with infinite residue fields, then $\gll_R(R)\les \gll_S(S).$
 \end{conjecture}

     By combining \cite[Theorem~2.1]{Ding:1994} and \cite[Corollary~5.2]{Hashimoto/Shida:1997}, the conjecture is already known to hold when $R\to S$ is a flat local extension of Gorenstein, strict Cohen-Macaulay rings having infinite residue fields and a regular fiber. We give a substantial improvement below. 

  \begin{theorem}
    Suppose $R\to S$ is a local extension of Cohen-Macaulay rings with infinite residue fields, and assume $\fdim_R(S)<\infty$. If $R$ is strict Cohen-Macaulay, then $\gll_R(R)\les \gll_S(S).$  In particular, \cref{conjecture_loewy_lech} holds in this setting. 
\end{theorem}

\begin{proof}
  By \cite[Lemma~3.3]{Hashimoto/Shida:1997} generalized Loewy length is invariant upon completion, so we can assume both $R$ and $S$ are complete. Now by \cite[Theorem~1.1]{Avramov/Foxby/Herzog:1994}, there exists a Cohen factorization of $\varphi$: 
     \[
     R\xra{\iota }R'\xra{\vp'}S
     \]
     where $\iota$ weakly regular (that is to say, $\iota$ is flat with regular fiber) and $\vp'$ is surjective. 

     As $\iota$ is weakly regular, the map $\iota^\g\colon R^\g\to (R')^\g$ is tangentially flat, in the sense that it is a flat extension of standard graded algebras with a symmetric algebra fiber; see \cite[Theorem~1.2]{Herzog:1991}. In particular, $R'$ is strict Cohen-Macaulay. As a consequence, the second equality below holds:
     \begin{equation}\label{e_r_to_r'}
     \gll_R(R)=\reg(R^\g)+1=\reg((R')^\g)+1=\gll_{R'}(R')\,;
     \end{equation}
     the outside equalities are from \cref{perfectModules}.
     Again, using that $\iota$ is weakly regular and $\fdim_R(S)<\infty$ we have that $\pdim_{R'}(S)<\infty$; see \cite[Lemma~3.2]{Avramov/Foxby/Herzog:1994}. 
     Also, since $S$ is Cohen-Macaulay it follows that it is a Cohen-Macaulay $R'$-module. Therefore, we have 
     \[
     \gll_S(S)\ges \gll_{R'}(R')=\gll_R(R)\,,
     \]
     where inequality uses \cref{mainTheorem}, and the equality comes from \cref{e_r_to_r'}.
 \end{proof}

With \cref{l_bound}, we obtain an immediate corollary in the following, where we refer the reader to the nice introduction on reductions and reduction numbers in~\cite[Chapter~8]{Huneke/Swanson:2006}.  Furthermore, the next corollary verifies a conjecture of Hanes~\cite[Conjecture~3.1]{Hanes:2001} when $R$ is strict Cohen-Macaulay:  \emph{If $(R, \m) \to (S, \n)$ is a local extension of Cohen-Macaulay rings with infinite residue fields, then $r(\m)\les r(\n)$.} Hanes had previously established the conjecture in the standard graded setting. 

 \begin{corollary}
 \label{cor_reduction_regularity_nondec}
      Suppose $(R, \m) \to (S, \n)$ is a local extension of Cohen-Macaulay rings with infinite residue fields, and assume $\fdim_R(S)<\infty$. If $R$ is strict Cohen-Macaulay, then $r(\m) \les r(\n)$ and $\reg(R^\g) \les \reg(S^\g)$.
 \end{corollary}

 \begin{remark}
 \label{c_IMW}
      The \emph{Length Conjecture} of Iyengar--Ma--Walker~\cite[Conjecture 1]{Iyengar/Ma/Walker:2022} implies  Lech's conjecture in full generality for Cohen-Macaulay rings; see also \cite[Chapter V]{Ma:2014}. The former posits that over a local ring $R$, any nonzero module of finite projective dimension has $\ell_R(M)\ges e(R)$. This is known to hold when $R$ is a strict complete intersection ring (that is, its associated graded is a complete intersection) or when $R$ is a localization of a standard graded algebra, at its homogeneous maximal ideal, over a perfect field of positive characteristic \cite{Iyengar/Ma/Walker:2022}. The Loewy Length Conjecture (\cref{introconj}) is the analog of the Length Conjecture; cf.\@ \cref{r_loewy}. Hence, it seems appropriate to ask the following.
 \end{remark}

 \begin{question}
     If the Loewy Length Conjecture holds for all Cohen-Macaulay rings having infinite residue field, does this imply \cref{conjecture_loewy_lech}  holds? 
 \end{question}

The only thing to determine to answer the question in the positive is whether the generalized Loewy length remains the same along a weakly regular map between Cohen-Macaulay rings having infinite residue field; without infinite residue fields this is false because of the example of Hashimoto--Shida~\cite{Hashimoto/Shida:1997}. 

\section{Bounds over general local rings}\label{s_ci}

The main result of the section is a stronger version of \cref{intro_ci} from the introduction; it establishes a lower bound for the Loewy length of a nonzero module of finite flat dimension over a complete intersection ring. This result can be found later in the section (see \cref{t_mark}). First, we improve known lower bounds for Loewy lengths over general local rings in \cref{t_ci}.

 \begin{chunk}
 \label{c_order}
  For a nonzero element $f$ in a local ring $(R,\m)$ recall that its $\m$-adic order, denoted $\ord(f)$, is the smallest positive integer $n$ such that $f\in \m^n\smallsetminus\m^{n+1}$; by Krull's Intersection Theorem this is well-defined, and we set $\ord(0)=\infty$. The \emph{order of R}, denoted $\ord(R)$, is the minimal $\m_Q$-adic order of the kernel of a minimal Cohen presentation $(Q,\m_Q)\to \widehat{R}$.
The \emph{max-order of} $R$, denoted $\max\ord(R)$, is the maximal $\m_Q$-adic order of a minimal generator of the kernel of a minimal Cohen presentation $(Q,\m_Q)\to \widehat{R}$. We always have inequalities 
    \[
    \ord(R)\les \max\ord(R)\les \reg(R^\g)+1\,,
    \]
     where the first is by definition. The second inequality holds since $\max\ord(f)$ is the degree of a minimal homogeneous generator for the kernel of a surjective map from a standard graded polynomial ring to $R^\g.$
\end{chunk}
 
 \begin{chunk}\label{c_bound_general}
     For a Gorenstein local ring $R$, if $M$ is a nonzero finite length module having finite projective dimension, then 
 \[
 \ll_R(M)\ges \ord(R)\,.
\]
 This was first established in the case that $R$ is also assumed to be strict Cohen-Macaulay with infinite residue field in \cite[Theorem~1.1]{Avramov/Buchweitz/Iyengar/Miller:2010}, and for general Gorenstein rings in \cite[Theorem~1.1]{Puthenpurakal:2017}. 
 \end{chunk}

The following result removes the Gorenstein hypothesis and improves the bound in \cref{c_bound_general} in a rather drastic way. The proof is similar to that of \cite[Proposition~1.2]{Koh/Lee:1998} which identifies homology classes in Tor-modules depending on the Loewy length of a module. 

 \begin{theorem}
 \label{t_ci}
      Assume $R$  is a local  ring, and fix a minimal Cohen presentation $\widehat{R}\cong Q/(f_1,\ldots,f_t)$ and a nonzero $R$-module $M$. 
     \begin{enumerate}
         \item If  $M$ is finitely generated and $\ll_R(M)<\ord(R),$ then $\cx_R(M)\ges t$.
         \item    If $M$ has finite flat dimension and $\m M\neq M$, then $\ll_R(M)\ges \max\ord(R).$
     \end{enumerate}

 \end{theorem}
Above, recall  the complexity of a finitely generated $R$-module $M$ is
\[
\cx_R(M)=\inf\{d\in \mathbb{N}\mid \beta_n^R(M)\les a n^{d-1}\text{ for some $a>0$ and all $n$}\}\,;
\]
this is a measure of the polynomial growth of the Betti numbers of $M$ over $R$. In particular $\cx_R(M)=0$ if and only if $\pdim_R(M)<\infty$.  See \cite[Chapter~5]{Avramov:2010} for more on this and other asymptotic homological invariants defined over local rings; see also \cite{Avramov/Iyengar/Miller:2006}.

\begin{proof}[Proof of Theorem \ref{t_ci}]
    We make use of a well-known dg algebra structure on a minimal free resolution of the residue field $k$ over $R$; see \cite[Chapter~6]{Avramov:2010}, or \cite{Gulliksen/Levin:1969}, for more details and any unexplained notation or terminology. 
    
    The essential point is that a minimal free resolution of $k$ over $R$ has the form $R\langle X\rangle$ where as a graded $R$-algebra it is the free divided power algebra with $X=X_1,X_2,X_3,\ldots$, and  $X_i$ consists of variables of degree $i$. The differential is determined by its values on the variables and extended via the Leibniz rule and respecting divided powers. We are particularly interested in the differential on the divided power subalgebra on $X_2$, which can be described explicitly as follows. 
    
There is no harm in assuming $R$ is complete and so using a minimal Cohen presentation $R=Q/(f_1,\ldots,f_t)$, fix a minimal set of generators $x_1,\ldots, x_d$ of $\m_Q$, and write 
    \[
    f_i=\sum \tilde{a}_{ij} x_j\quad \text{with }\tilde{a}_{ij}\in \m_Q\,.
    \]
    Then $X_1=\{e_1,\ldots,e_d\}$, $X_2=\{y_1,\ldots,y_t\}$ and we have 
    \[
    \del(y_i)=\sum a_{ij} e_j\quad\text{ for each }i=1,\ldots, t
    \]
    where $a_{ij}$ is the image of $\tilde{a}_{ij}$ in $\m$, the maximal ideal of $R$; this calculation is classical and due to Tate~\cite[Theorem~4]{Tate:1957}, but it is also explained in the references above. As a consequence, the differential on the divided power monomial $y_1^{(c_1)} \cdots  y_t^{(c_t)}$ is 
    \begin{equation}\label{e_differential}
     \del(y_1^{(c_1)}\cdot \cdots \cdot y_t^{(c_t)})=\sum_{i=1}^t \sum_{j=1}^da_{ij}e_jy_1^{(c_1)}\cdots y_i^{(c_i-1)}\cdots y_t^{(c_t)}\,.   
    \end{equation}

    Now assume $\ll_R(M)<\ord(R)=n$, then in the Cohen presentation above each $f_i\in \m_Q^n$ and so we can assume $a_{ij}\in \m^{n-1}$. In particular, from \cref{e_differential} it follows that  \[\del(y_1^{(c_1)}\cdot \cdots \cdot y_t^{(c_t)}\otimes m)=\sum_{i=1}^t \sum_{j=1}^da_{ij}e_jy_1^{(c_1)}\cdots y_i^{(c_i-1)}\cdots y_t^{(c_t)}\otimes m=0
    \]
    as each $a_{ij}\in \m^{n-1}\subseteq \ann_R(M).$ Hence, for any $m\in M\smallsetminus \m M$, we have a cycle 
    \[
    y_1^{(c_1)}\cdots  y_t^{(c_t)}\otimes m
    \]
    that cannot be a boundary as $\del(R\langle X\rangle \otimes_R M)\subseteq \m R\langle X\rangle \otimes_R M$. Therefore, we have obtained the inequality below: 
    \[
    \beta^R_i(M)=\rank_k\Tor^R_i(k,M)=\rank_k \H_i(R\langle X\rangle \otimes_R M) \ges \rank_k (\Gamma_i\otimes_k M/\m M)
    \]
    where $\Gamma=k\langle X_2\rangle$, the free divided power algebra on the degree two variables $X_2$. The Hilbert series for $\Gamma$ (over $k$) is $(1-z^2)^{-t}$ and so the desired result on complexity has been established.  

     Next, assume $\ll_R(M)<\max\ord(R)=n$, then in the Cohen presentation above at least one $f_i$ belongs to  $\m^n_Q$ and so for that $i$, we can assume  $a_{ij}\in \m^{n-1}$ for each $j$. In this case, from \cref{e_differential} it follows that  \[\del(y_i^{(c)}\otimes m)= \sum_{j=1}^ta_{ij}e_jy_i^{(c)}\otimes m=0
    \]
    as each $a_{ij}\in \m^{n-1}\subseteq \ann_R(M).$ Again, the minimality of $R\langle X\rangle\otimes_R M$ gives us cycles that are not boundaries in each even degree: 
    \[
    \{y_i^{(c)}\otimes m\mid c\ges 0\text{ and }m\in M\smallsetminus \m M\}\,.
    \]
   Finally, recalling the homology of $R\langle X\rangle\otimes_R M$ is $\Tor^R(k,M)$ we have shown that $M$ has infinite flat dimension. 
\end{proof}

\begin{remark}
\label{r_extremal}
     \cref{t_ci} shows that when $R$ is a local complete intersection ring and $M$ is a nonzero finitely generated $R$-module with $\ll_R(M)<\ord(R)$, then $M$ is extremal in the sense of Avramov~\cite{Avramov:1996}. Roughly speaking, this means the Betti sequence of $M$ grows at the same rate as the Betti sequence of the residue field (the latter has maximal growth, in a precise sense, among all finitely generated $R$-modules). 
    In light of this and \cref{t_ci}, we ask the following. 
\end{remark}
\begin{question}
\label{q_extremal}
    For a local ring  $R$, if a nonzero finitely generated $R$-module $M$ satisfies $\ell\ell_R(M) < \ord(R)$, must $M$ be extremal? 
\end{question}

The remainder of the article centers around establishing a lower bound for Loewy lengths over complete intersection rings. We give two proofs: the first proof uses similar ideas to the ones from the proof of \cref{mainTheorem}, while the second proof, suggested to us by Mark Walker, uses a construction also used to establish the Length Conjecture for strict complete intersection rings in \cite[Corollary~V.29]{Ma:2014}.

\begin{theorem}
\label{t_mark}
    Assume $R$ is a complete intersection ring of codimension $c$. If $M$ is a nonzero $R$-module of finite flat dimension, then 
    \[
    \ll_R(M)\ges \sum_{i=1}^c\ord(f_i)-c+1
    \]
    where $\widehat{R}\cong Q/(f_1,\ldots,f_c)$ is any minimal Cohen presentation of $R$. In particular, if $M$ is a nonzero Cohen-Macaulay module of finite projective dimension, then \[\gll_R(M)\ges \sum_{i=1}^c\ord(f_i)-c+1\,.\]
\end{theorem}

The bound in \cref{t_mark} can be lower than $\gll_R(R)$ when $R$ is a local complete intersection ring that is not strictly Cohen-Macaulay. 

\begin{example}
\label{r_alessandro}
In \cite{Alessandro:2016}, DeStefani showed that the following one dimensional complete intersection $k$-algebra, with $k$ a field, 
    \[ R = \frac {k\llbracket  x, y, z \rrbracket }{(x^2-y^5, xy^2+yz^3-z^5)}\] 
    has $\gll_R(R)=6$ and a calculation, using \texttt{Macaulay2}, shows $\reg(R^\g)=6$. In particular, $R$ is not strict Cohen-Macaulay; one can also see this from a direct calculation (and analysis) of $R^\g$. In fact, by \cref{mainTheorem}, it follows that $R$ does not have minimal regularity (in the sense defined in \cref{remark_min_reg}). Here, the lower bound from \cref{t_mark} is 4.
\end{example}

\begin{example}\label{e_nonstrictci}
    Fix $n\ges 2,$ and let $R=k\llbracket a,b,x,y,z\rrbracket/(a^2-x^n,ab-y^n,b^2-z^n)$ where $k$ is a field. Using that $(x,z)$ is a minimal reduction, one can argue $R^\g$ is a free extension over $k[x^*,z^*]$; as a consequence, $R$ is a complete intersection ring that is strict Cohen-Macaulay. In particular, the bound from \cref{new_main_theorem} implies 
    \[
    \ll_R(M)\ges \ll_R(R/(x,z))=\ll(k\llbracket a, b, y\rrbracket/(a^2,ab-y^n,b^2))=2n\,
    \]
    for any nonzero $R$-module having finite flat dimension. However, the bound from \cref{t_mark} only asserts that $\ll_R(M)$ is at least 4, which is arbitrarily smaller than $\gll_R(R)$ as $n$ grows. 
\end{example}

\begin{remark}
    In \cref{t_mark}, each $\ord(f_i)$ is at least two  and so \cref{t_mark} implies $\ll_R(M)\ges c+1$. Hence the bound from \cref{t_mark} strengthens the already known bounds for \emph{modules} over complete intersection rings in \cite{Avramov/Buchweitz/Iyengar/Miller:2010}; the bounds from the latter are known to hold for the sum of the Loewy lengths of the homology modules of perfect complexes, and is tight as the Koszul complex $K$ on the maximal ideal always has $\sum\ll_R\H_i(K)=c+1.$ 
\end{remark}

\begin{proof}[First proof of \cref{t_mark}]
    First, we need a bit of notation; see \cite{Avramov:2010} for background on dg modules in commutative algebra. 
    
    We can assume $R$ is complete and hence  $R\cong Q/(f_1,\ldots,f_c)$ with $f_1,\ldots,f_c$ a regular sequence in the regular ring $(Q,\m_Q)$. Let $\m_Q=(t_1,\ldots,t_c,x_1,\ldots,x_d)$ where the image of $x_1,\ldots,x_d$ is a maximal regular sequence in $\m_R\smallsetminus\m_R^2$, and write 
    \[
    f_i=\sum_{j=i}^ca_{ij}t_j+\sum_{j=1}^db_{ij}x_j
    \]
    with each $a_{ij}$ and $b_{ij}$ belonging to $\m_Q^{\ord(f_i)-1}.$ Set 
    \begin{align*}
    E&=\mathrm{Kos}^Q(\bm{f})=Q\langle e_1,\ldots, e_c\mid \del e_i=f_i\rangle\\
    \tilde{A}&=E\otimes_Q \mathrm{Kos}^Q(\x)=Q\langle e_1,\ldots,e_c,e'_1,\ldots, e'_d\mid \del e_i=f_i\,, \del e_i'=x_i\rangle\\
    K&=\mathrm{Kos}^Q(\bm{t},\x)=Q\langle e_1'',\ldots,e_c'',e'_1,\ldots, e'_d\mid \del e_i''=t_i\,, \del e_i'=x_i\rangle\,,
    \end{align*}
    which are dg $E$-modules; the $E$-actions on $E,\tilde{A}$ are the obvious ones and the $E$ action on $K$ is given by the map on dg $Q$-algebras $ E\to K$ determined by  
    \[
    e_i\mapsto \sum_{j=i}^ca_{ij}e_j''+\sum_{j=1}^db_{ij}e_j'\,.
    \]
    Finally, consider the map $\alpha\colon \tilde{A}\to K$ of dg $E$-modules extending the map above and sending each $e_j'$ in $\tilde{A}$ to $e_j'$ in $K$. 

Note that $\alpha_i=\bigwedge^i \alpha_1\colon \tilde{A}_i\to K_i$ and now using that each $a_{ij}$, $b_{ij}$ is in $\m_Q^{\ord(f_i)-1}$, a direct calculation shows
\begin{equation}
    \label{e_koszul_containment}
    \alpha(E_c\otimes_Q \mathrm{Kos}^Q(\bm{x}))\subseteq \m_Q^n K\quad \text{where}\quad n= \sum_{i=1}^c \ord(f_i) -c\,.
\end{equation}
 Also, it is well known (see, for instance, \cite[Exercise~21.23]{Eisenbud:1995}) that the cone of 
\[
\shift^{c+d}\alpha^\vee\colon \shift^{c+d}K^\vee \to \shift^{c+d}\tilde{A}^\vee\,,
\] 
is a $Q$-free resolution of $R/(\x,\det(a_{ij}))$ where $(-)^\vee=\Hom_Q(-,Q)$. Using the self-duality of Koszul complexes, $\shift^{c+d}\alpha^\vee$ can be regarded as a dg $E$-module map from $K$ to $\tilde{A}$; write $\tilde{\sigma}\colon K\to \tilde{A} $ for this map using these identifications of each Koszul complex with its dual. By \cref{e_koszul_containment}, composing $\tilde{\sigma}$  with the projection of $\tilde{A}\to E_0\otimes_Q \mathrm{Kos}^Q(\x)$ factors as the map of graded $Q$-modules:
\begin{equation}\label{e_koszul_cone}
    \begin{tikzcd}
       K\ar[r,"\tilde{\sigma}"] \ar[d,dashed] &\tilde{A}\ar[d,two heads] \\
     \m_Q^nE_0\otimes_Q \mathrm{Kos}^Q(\x) \ar[r,hook] &  E_0\otimes_Q \mathrm{Kos}^Q(\x)\,;
    \end{tikzcd}
\end{equation}
here we used that the identification of $\shift^{c+d}\tilde{A}$ with $\tilde{A}$, restricts to an isomorphism
\[\shift^{c+d}(E_c\otimes_Q\mathrm{Kos}^Q(\x))^\vee\cong \shift^cE_c^\vee\otimes_Q \shift^d( \mathrm{Kos}^Q(\x))^\vee\cong E_0\otimes_Q \mathrm{Kos}^Q(\x)\,.\] 

Now letting $\Gamma\coloneqq R\langle y_1,\ldots,y_c\rangle$ where each $y_i$ is a degree two divided power variable, $1\otimes\tilde{\sigma}\colon \Gamma\otimes_Q^\tau K\to \Gamma\otimes_Q^\tau \tilde{A}$ is a map of $R$-complexes;
here, for a dg $E$-module $N$, the $R$-complex $\Gamma\otimes_Q^\tau N$ is the construction of Eisenbud~\cite{Eisenbud:1980} and Shamash~\cite{Shamash:1969} (see also \cite[Section~2]{Avramov/Buchweitz:2000a}). Explicitly, $\Gamma\otimes_Q^\tau N$  is the free graded $R$-module $\Gamma\otimes_Q N$ with differential 
\[
y_1^{(h_1)}\cdots y_c^{(h_c)}\otimes n\mapsto y_1^{(h_1)}\cdots y_c^{(h_c)}\otimes \del^N(n)+\sum_{i=1}^cy_1^{(h_1)}\cdots y_i^{(h_i-1)}\cdots y_c^{(h_c)}\otimes e_in\,.
\]
Set $F=\Gamma\otimes_Q^\tau K$, which by \cite[Theorem~2.4]{Avramov/Buchweitz:2000a} (see also \cite{Tate:1957}), is a minimal free resolution of $k$ over $R$. Also, there is an isomorphism of complexes 
\[
\Gamma\otimes_Q^\tau \tilde{A} \cong R\langle e_1,\ldots,e_c,e_1',\ldots, e_d',y_1,\ldots,y_c\mid \del e_i=0,\, \del e'_i=x_i,\, \del y_i=e_i\rangle 
\]
and so we have quasi-isomorphisms
\[
\Gamma\otimes_Q^\tau \tilde{A}\xra{\simeq} R\otimes_E \tilde{A} \cong \mathrm{Kos}^R(\x)\xra{\simeq} R/(\x)\,.
\]

Set $A=R\otimes_E \tilde{A}$ and let $\sigma$ be the composition of chain maps
\[
F=\Gamma\otimes_Q^\tau K\xra{1\otimes\tilde{\sigma}}  \Gamma\otimes_Q^\tau \tilde{A}\xra{\simeq}  A\,.
\] 
Forgetting differentials,  $\sigma$ factors as the following map of  graded $Q$-graded modules:
\begin{equation*}
    \begin{tikzcd}
       F\ar[rr,"1\otimes  \tilde{\sigma}"] \ar[d,two heads] \ar[rrd,"\sigma"]& & \Gamma\otimes_Q^\tau \tilde{A}\ar[d,two heads] \\
    R\otimes_Q K \ar[r,"1\otimes \tilde{\sigma}"] &  R\otimes_Q \tilde{A}\ar[r,two heads] & R\otimes_E\tilde{A}=A
    \end{tikzcd}
\end{equation*}
and since the surjection $R\otimes_Q \tilde{A}\twoheadrightarrow R\otimes_E \tilde{A}$ factors through $(R\otimes_Q E_0)\otimes_Q \tilde{A}$ we can apply \cref{e_koszul_cone} to conclude that $\sigma(F)\subseteq \m^n_R A.$  Hence,  \cref{newlemma} implies that $\ll_R(M)\ges n+1.$

It remains to repeat the argument from \cref{mainTheorem} to establish the desired lower bound for perfect modules of positive dimension. 
\end{proof}

Before presenting the second proof of \cref{t_mark}, we need the following construction from \cite{Herzog/Ulrich/Backelin:1991}; see also \cite[Theorem~V.27]{Ma:2014}.

\begin{chunk} \label{c_MCMfiltration}
Let $(Q, \m)$ be a Cohen-Macaulay ring and $f \in \m^d$ a regular element. Then for some integer $s \ges 1$, there exists a  filtration of $R=Q/(f)$-modules 
    \[ 0 = U_0 \subseteq U_1 \subseteq \ldots \subseteq U_d = R^s\] such that $U_{i-1} \subseteq \m U_i$ and each $U_i/U_{i-1}$ is a maximal Cohen-Macaulay (abbreviated, as usual, to MCM) $R$-module having finite projective dimension over $Q$. The condition on $U_i/U_{i-1}$ is equivalent to $\pdim_Q (U_i/U_{i-1}) = 1$.
\end{chunk}

\begin{lemma} \label{l_filtration}
    Suppose $(Q, \m)$ is a Cohen-Macaulay ring and $R = Q/(f_1, \ldots, f_c)$ where $f_1, \ldots, f_c$ is a $Q$-regular sequence. Set 
    \[n= \sum_{i=1}^c \ord(f_i) -c +1\,.\] Then for some integer $t \ges 1$, there exists a filtration of $R$-modules \[ 0 = U_0 \subseteq U_1 \subseteq \ldots \subseteq U_n = R^t\] such that $U_{i-1} \subseteq \m U_i$ and $U_i/U_{i-1}$ is an MCM $R$-module for $1\les i \les n.$
\end{lemma}
\begin{proof}
    We induct on $c$. The base case is \cref{c_MCMfiltration}. Let $R' = Q/(f_1, \ldots, f_{c-1})$ and \[n' = \sum_{i=1}^{c-1} \ord(f_i) - (c-1) +1 .\]
    By the induction hypothesis, for some $t'\ges 1$, we have a filtration of $R'$-modules 
    \[0 = U'_0 \subseteq U'_1 \subseteq \ldots \subseteq U'_{n'} \cong (R')^{t'}
    \]where $U'_{i-1} \subseteq \m U'_i$ and $U'_i/U'_{i-1}$ is an MCM $R'$-module. Since $R'$ is Cohen-Macaulay and $f=f_c$ regular on $R'$ in $\m^{\ord(f)}R'$,  we can apply \cref{c_MCMfiltration} to obtain a filtration of $R=R'/(f)$-modules 
    \[0=V_0 \subseteq V_1 \subseteq \cdots \subseteq V_{\ord(f)} \cong R^t\] for some $t\ges 1$ where each $V_{i-1} \subseteq \m V_i$ and  $\pdim_{R'}(V_i/V_{i-1})=1$; in particular, $\pdim_{R'}V_1 = 1$. 
    
    We claim that the following filtration of $R$-modules of length $n$ has all the desired properties: \[  U'_0 \otimes V_1 \subseteq U'_1 \otimes V_1 \subseteq \cdots \subseteq U'_{n'} \otimes V_1 \subseteq U'_{n'} \otimes V_2 \subseteq \cdots \subseteq U'_{n'} \otimes V_{\ord(f)}\,.\]
    Indeed, $U_0'\otimes V_1=0$ and $U'_{n'}\otimes V_{\ord(f)}\cong R^{tt'}$, and for $i\ges 1$ every term in the filtration is of the form 
    \[
    U_{n'}'\otimes V_i\cong (R')^{t'}\otimes V_i\cong V_i^{ t'}\,;
    \]
    as a consequence, there are the desired containments $U_{n'}'\otimes V_{i-1}\subseteq \m U_{n'}\otimes V_i$  and 
    \[
    \pdim_Q(U_{n'}'\otimes V_i/ U_{n'}'\otimes V_{i-1})=\pdim_Q(V_i^{ t'}/V_{i-1}^{t'})=c
    \]
    where the last equality used  $\pdim_{R'}(V_i/V_{i-1})=1;$ thus each of these subquotients are MCM $R$-modules.   Therefore, it remains to verify the first $n'$-steps satisfy the desired properties as well. 

    To this end, for each $i$, as $U'_i/U'_{i-1}$ is MCM over $R'$ and $\pdim_{R'}V_1 < \infty$ it follows that $\Tor^{R'}_{>0}(U'_i/U'_{i-1},V_1)=0$ and hence the maps along the bottom are injective:
    \[\begin{tikzcd}
	&& {\m U'_i \otimes_{R'} V_1} \\
	{} & {U'_{i-1} \otimes_{R'} V_1} && {U'_i \otimes_{R'} V_1}
	\arrow[from=2-2, to=1-3]
	\arrow[from=1-3, to=2-4]
	\arrow[hook, from=2-2, to=2-4]
\end{tikzcd}\]and so  $U'_{i-1} \otimes V_1 \subseteq \m U'_i \otimes V_1 \subseteq \m(U'_i \otimes V_1)$. It only remains to observe
\begin{align*}
\depth( U'_i/U'_{i-1}) &= \depth (U'_i/U'_{i-1} \otimes^{\mathsf L}_{R'} V_1) + \pdim_{R'}(V_1) \\
&= \depth (U'_i/U'_{i-1} \otimes_{R'} V_1) + 1\\
&=\depth(U'_i\otimes V_1/U'_{i-1}\otimes V_1)+1
\end{align*}
where the first equality uses the derived depth formula \cite[Corollary 2.2]{Iyengar:1999}, and second equality uses that $\pdim_{R'}(V_1)=1$ and $\Tor^{R'}_{>0}(U'_i/U'_{i-1},V_1)=0$.
We need only note that $U_i'/U_{i-1}'$ is an MCM $R'$-module to deduce that $U'_i\otimes V_1/U'_{i-1}\otimes V_1$ is MCM over $R$. 
\end{proof}

\begin{lemma}\label{tor_ind}
    Assume $M$ is an $\m$-power torsion module of finite flat dimension and $N$ is MCM, then $\Tor^R_{\geqslant 1}(M,N)=0$.
\end{lemma}
\begin{proof}
 Let $t=\dim R$ and observe that we have isomorphisms
    \[
    M\lotimes_R N\simeq \mathrm{R}\Gamma_\m(M)\lotimes_R N\simeq M\lotimes_R\mathrm{R}\Gamma_\m(N)\simeq M\lotimes\shift^{-t}\H^{t}_\m(N)\,;
    \]
    the first isomorphism uses that $M$ is $\m$-power torsion, the second is straightforward, and the third uses that $N$ is an MCM $R$-module. Finally as $t\ges d=\depth(R)$, and since $M$ has finite flat dimension at most $d$, it follows that $\Tor^R_{\geqslant 1}(M,N)=0.$
\end{proof}

\begin{proof}[Second proof of \cref{t_mark}]
    We can assume $R\cong Q/(f_1,\ldots,f_c)$ with $f_1,\ldots,f_c$ a regular sequence in the regular ring $Q$. By \cref{l_filtration}, there exists a filtration by $R$-submodules 
    \[ 0 = U_0 \subseteq U_1 \subseteq \ldots \subseteq U_n = R^t\,,\quad \text{with}\quad n= \sum_{i=1}^c \ord(f_i) -c +1\,,\] 
    such that each $U_{i-1} \subseteq \m U_i$ and $U_i/U_{i-1}$ is an MCM $R$-module, and $t\ges 1$. 
    
    If $M$ has infinite Loewy length there is nothing to show, so assume it is finite.  Arguing as in the proof of \cref{l_filtration}, since $M$ is $\m$(-power) torsion of finite flat dimension and each subquotient of the filtration is an MCM $R$-module we can apply \cref{tor_ind} to obtain inclusions $U_{i-1} \otimes M \subseteq  \m(U_i \otimes M)$ for each $i$. In particular, 
    \[0 \neq U_1 \otimes M \subseteq \m (U_2\otimes M)\subseteq \ldots  \subseteq \m^{n-1} (U_n\otimes M)=\m^{n-1} M^t\,,\]
    and thus, $\ll_R(M) \ges n$. 
    
    Finally,  it remains to repeat the argument from \cref{perfectModules} to establish the desired lower bound for perfect modules of positive dimension. 
\end{proof}

In fact, the second proof establishes a bound in the relative setting. That is to say, when a Cohen-Macaulay local ring $R$ admits a deformation, there is the following uniform lower bound on the Loewy lengths of nonzero modules of finite flat dimension.

\begin{theorem}
    Assume a Cohen-Macaulay local ring $R$ is a deformation: 
    \[ R\cong Q/(f_1,\ldots,f_c)\quad\text{with}\quad f_1,\ldots,f_c\text{ a regular sequence.}
    \]
    For any  nonzero $R$-module $M$ of finite flat dimension, we have 
    \[
    \ll_R(M)\ges \sum_{i=1}^c\ord(f_i)-c+1\,. \qed
    \]
\end{theorem}

\bibliographystyle{amsplain}
\bibliography{references}

\end{document}